\documentclass[12pt]{amsart}
\usepackage{amssymb, hyperref, fullpage, mathrsfs,xcolor}

\newtheorem{theorem}{Theorem}[]
\newtheorem{proposition}[theorem]{Proposition}

\begin{document}
\title{The density of ramified primes}
\author{Jyoti Prakash Saha}
\address{Department of Mathematics, Ben-Gurion University of the Negev, Be'er Sheva 8410501, Israel}
\email{jyotipra@post.bgu.ac.il}



\begin{abstract}
Let $F$ be a number field, $\mathcal{O}$ be a domain with fraction field $\mathcal{K}$ of characteristic zero and $\rho: \mathrm{Gal}(\overline F/F) \to \mathrm{GL}_n(\mathcal{O})$ be a representation such that $\rho\otimes\overline{\mathcal{K}}$ is semisimple. If $\mathcal{O}$ admits a finite monomorphism from a power series ring with coefficients in a $p$-adic integer ring (resp. $\mathcal{O}$ is an affinoid algebra over a $p$-adic number field) and $\rho$ is continuous with respect to the maximal ideal adic topology (resp. the Banach algebra topology), then we prove that the set of ramified primes of $\rho$ is of density zero. If $\mathcal{O}$ is a complete local Noetherian ring over $\mathbb{Z}_p$ with finite residue field of characteristic $p$, $\rho$ is continuous with respect to the maximal ideal adic topology and the kernels of pure specializations of $\rho$ form a Zariski-dense subset of $\mathrm{Spec} \mathcal{O}$, then we show that the set of ramified primes of $\rho$ is of density zero. These results are analogues, in the context of big Galois representations, of a result of Khare and Rajan, and are proved relying on their result.
\end{abstract}

\maketitle

\section{Introduction}
The study of continuous representations of the absolute Galois groups of number fields with coefficients in $p$-adic number fields is an important theme in number theory. For each normalized eigen cusp form of weight at least two, there exists an associated Galois representation of the absolute Galois group $G_{\mathbb{Q}}$ of $\mathbb{Q}$ with coefficients in $\overline{\mathbb{Q}}_p$ by the works of Eichler \cite{Eichler54}, Shimura \cite{ShimuraCorrespondances} and Deligne \cite{DeligneModFormAndlAdicRepr}. In 1980's, Hida constructed $p$-adic analytic families of ordinary modular forms \cite{HidaGalrepreord, HidaIwasawa}. In particular, he proved that there exists a big Galois representation (i.e. a Galois representation of $G_{\mathbb{Q}}$ with coefficients in a finite extension of the fraction field of $\mathbb{Z}_p[[1+p\mathbb{Z}_p]]\simeq \mathbb{Z}_p[[X]]$) interpolating the traces of the Galois representations attached to the $p$-ordinary normalized eigen cusp forms of a given tame level. It is known that the Galois representations associated with the normalized eigen cusp forms are unramified almost everywhere and so is Hida's big Galois representation. 

The \'etale cohomologies of smooth projective varieties over number fields with coefficients in $\mathbb{Q}_p(r)$ are examples of continuous representations of the absolute Galois groups of number fields. These representations are unramified almost everywhere. However, there exist continuous, reducible, indecomposable representations which are ramified everywhere (see \cite[III-12]{SerreAbelianEllAdic}). In \cite{RamakrishnaInfinitelyRamified}, Ramakrishna constructed examples of two dimensional surjective representations of $G_{\mathbb{Q}}$ over $\mathbb{Z}_p$ which are ramified at infinitely many primes. One may also consider representations of $G_{\mathbb{Q}}$ with coefficients in extensions of the fraction field of $\mathbb{Z}_p[[X]]$ which are ramified at infinitely many primes. Some examples of such representations can be obtained by taking extension of scalars of the representations constructed by Ramakrishna.

Khare and Rajan proved in \cite{KhareRajanRamification} that the set of ramified primes of a continuous semisimple representation of the absolute Galois group of a number field over a $p$-adic number field is of density zero. One may ask if the set of ramified primes of the representations of the absolute Galois groups of number fields with coefficients in integral domains of large Krull dimension (for example, $\mathbb{Z}_p[[X_1, \cdots, X_m]]$ or $\mathbb{Q}_p\langle X_1, \cdots, X_m\rangle$) is of density zero. The aim of this article is to answer this question under suitable assumptions. 

\section{Results obtained}
In the following, $\mathcal{O}$ denotes an integral domain of characteristic zero with fraction field $\mathcal{K}$. Fix an algebraic closure $\overline{\mathcal{K}}$ of $\mathcal{K}$. Let $\rho: G_F \to \mathrm{GL}_n(\mathcal{O})$ be a representation of the absolute Galois group $G_F:= \mathrm{Gal}(\overline F/F)$ of a number field $F$ such that $\rho\otimes \overline{\mathcal{K}}$ is semisimple. We prove the following results about the set of ramified primes of $\rho$.

\begin{theorem}
\label{Thm}
If $\mathcal{O}$ admits a finite monomorphism from a power series ring with coefficients in a $p$-adic integer ring (resp. $\mathcal{O}$ is an affinoid algebra over a $p$-adic number field) and $\rho$ is continuous with respect to the maximal ideal adic topology (resp. the Banach algebra topology), then the set of ramified primes of $\rho$ is of density zero. 
\end{theorem}

We prove Propositions \ref{Prop}, \ref{Prop semisimple} in the next section and then give the proof of Theorem \ref{Thm}. Proposition \ref{Prop} states that given countably many nonzero elements of a ring $\mathcal{O}$ as in Theorem \ref{Thm}, $\mathcal{O}$ admits a continuous $\overline{\mathbb{Q}}_p$-specialization which vanishes at none of the given elements. Under some minor assumptions on $\mathcal{O}$, Proposition \ref{Prop semisimple} states that $\rho$ has enough semisimple $\overline{\mathbb{Q}}_p$-specializations. Applying Proposition \ref{Prop} to some nonzero entries (if any) of the monodromies of $\rho$ at the places of ramification and using Proposition \ref{Prop semisimple}, it follows that $\rho$ has a continuous semisimple specialization over $\overline{\mathbb{Q}}_p$ such that the set of ramified primes of $\rho$ coincides with the set of primes of ramification of this specialization. Then applying \cite[Theorem 1]{KhareRajanRamification}, Theorem \ref{Thm} follows. 

Theorem \ref{Thm} would also hold when $\mathcal{O}$ is a complete local Noetherian domain if Proposition \ref{Prop} holds for such rings. This proposition is proved for affinoid algebras (which are defined as the quotients of Tate algebras) by showing that it holds for Tate algebras and using the fact that affinoid algebras admit finite monomorphisms from Tate algebras (by Noether normalization \cite[Proposition 3, p. 32]{BoschFormalRigidGeo}). For complete local Noetherian rings of characteristic zero and residue characteristic $p$, it is known that such rings are quotients of the power series rings with coefficients in $p$-adic integer rings by the Cohen structure theorem \cite[Tag 032A]{StacksProject}. But we do not know if the quotients of power series rings over $p$-adic integer rings admit finite monomorphisms from power series rings. For this reason, we could not prove that Theorem \ref{Thm} holds even when $\mathcal{O}$ is a complete local Noetherian ring. However, under an additional assumption on the existence of enough pure specializations of $\rho$, we have the result below. Its proof relies on Proposition \ref{Prop semisimple} which guarantees the existence of enough semisimple $\overline{\mathbb{Q}}_p$-specializations of $\rho$, the fact that non-trivial action of inertia does not become trivial under pure specializations \cite[Theorem 3.1]{BigPuritySubAIF} and that the set of ramified primes of a continuous semisimple $p$-adic Galois representation is of density zero \cite[Theorem 1]{KhareRajanRamification}. 

\begin{theorem}
\label{Thm Pure}
If $\mathcal{O}$ is a complete local Noetherian ring over $\mathbb{Z}_p$ with finite residue field of characteristic $p$, $\rho$ is continuous with respect to the maximal ideal adic topology and the kernels of pure specializations of $\rho$ form a Zariski-dense subset of $\mathrm{Spec} \mathcal{O}$, then the set of ramified primes of $\rho$ is of density zero.
\end{theorem}

Theorems \ref{Thm}, \ref{Thm Pure} are analogues of the result \cite[Theorem 1]{KhareRajanRamification} of Khare and Rajan in the context of big Galois representations, i.e., in the context of Galois representations with coefficients in rings of large Krull dimension.

\section{Proofs of the results}

We prove the following propositions and then use them to establish Theorems \ref{Thm}, \ref{Thm Pure}.
\begin{proposition}
\label{Prop}
Suppose $\mathcal{O}$ admits a finite monomorphism from a power series ring with coefficients in a $p$-adic integer ring or $\mathcal{O}$ is an affinoid algebra over a $p$-adic number field. Then given countably many nonzero elements of $\mathcal{O}$, there exists a continuous $\mathbb{Z}_p$-algebra homomorphism $\lambda:\mathcal{O} \to \overline{\mathbb{Q}}_p$ which vanishes at none of them. 
\end{proposition}

\begin{proof}
Note that affinoid algebras admit finite monomorphisms from Tate algebras by Noether normalization \cite[Proposition 3, p. 32]{BoschFormalRigidGeo}. Hence it suffices to prove the proposition for power series rings and Tate algebras. 

Let $K$ be a $p$-adic number field with ring of integers $\mathcal{O}_K$. Let $\mathfrak{m}_K$ denote the maximal ideal of $\mathcal{O}_K$. Note that given a nonzero element of the Tate algebra $K\langle X_1, \cdots, X_m \rangle$, it has a nonzero multiple with Gauss norm one, which is divisible by $X_m - (a + a^2X_1 + a^3 X_2 + \cdots + a^mX_{m-1})$ only for finitely many elements $a\in \mathfrak{m}_K$ by the Weierstass division formula \cite[Theorem 8, p. 17]{BoschFormalRigidGeo}. Since $\mathfrak{m}_K$ contains $p\mathbb{Z}_p$ and $\mathbb{Z}_p$ is uncountable, there exists an element $a\in \mathfrak{m}_K$ such that each of the given countably many elements of $K\langle X_1, \cdots, X_m\rangle$ have nonzero images under the $K$-algebra map $K\langle X_1, \cdots, X_m\rangle \to K\langle X_1, \cdots, X_{m-1}\rangle$ sending $X_m$ to $a + a^2X_1 + a^3 X_2 + \cdots + a^mX_{m-1}$ and $X_i$ to $X_i$ for $i< m$. Hence the proposition holds for the Tate algebra $K\langle X_1, \cdots, X_m\rangle$ if it holds for $K\langle X_1, \cdots, X_{m-1}\rangle$. Since the proposition holds for $K$, by induction, it holds for any Tate algebra. 

Note that a nonzero element of $\mathcal{O}_K[[X_1, \cdots, X_m]]$ is divisible by $X_m - (a + a^2X_1 + a^3 X_2 + \cdots + a^mX_{m-1})$ only for finitely many elements $a\in \mathfrak{m}_K$. Since $\mathfrak{m}_K$ is uncountable, for some element $a\in \mathfrak{m}_K$, the given countably many elements of $\mathcal{O}_K[[X_1, \cdots, X_m]]$ have nonzero images under the $\mathcal{O}_K$-algebra map $\mathcal{O}_K[[X_1, \cdots, X_m]] \to \mathcal{O}_K[[X_1, \cdots, X_{m-1}]]$ sending $X_m$ to $a + a^2X_1 + a^3 X_2 + \cdots + a^mX_{m-1}$ and $X_i$ to $X_i$ for $i< m$. Then by induction, the proposition follows for power series rings. 
\end{proof}

\begin{proposition}
\label{Prop semisimple}
Suppose $\mathcal{O}$ is a complete local Noetherian ring over $\mathbb{Z}_p$ with finite residue field of characteristic $p$ or an affinoid algebra over a $p$-adic field. Then there exists a nonzero element $h\in \mathcal{O}$ such that $\lambda\circ \rho$ is semisimple for any homomorphism $\lambda: \mathcal{O} \to \overline{\mathbb{Q}}_p$ with $\lambda(h)\neq 0$. 
\end{proposition}

\begin{proof}
Note that a complete local Noetherian domain is a Nagata ring by \cite[Corollary 2, p. 234]{MatsumuraCommutativeAlgebra} and in particular, it is N-2, i.e., its integral closure in any finite extension of its fraction field is finite and hence Noetherian. If an affinoid algebra is an integral domain, then its integral closure in its fraction field is Noetherian (being finite by \cite[Theorem 3.5.1]{FresnelVanDerPutRAGApplication}) and hence its integral closure in any finite extension of its fraction field is also Noetherian by \cite[Proposition 31.B]{MatsumuraCommutativeAlgebra}. So the integral closure of $\mathcal{O}$ in any finite extension of $\mathcal{K}$ is Noetherian. 

Let $\mathcal{T}$ denote the free $\mathcal{O}$-module $\mathcal{O}^n$ with $G_F$ acting on it by $\rho$. We claim that there exists a finite extension $\mathscr{K}$ of $\mathcal{K}$ contained in $\overline{\mathcal{K}}$, a nonzero element $f$ in the integral closure $\mathscr{O}$ of $\mathcal{O}$ in $\mathscr{K}$ and representations $\rho_1 : G_ F\to \mathrm{GL}_{n_1} (\mathscr{O}_f), \cdots, \rho_r : G_ F\to \mathrm{GL}_{n_r} (\mathscr{O}_f)$ such that $\rho_i\otimes \overline{\mathcal{K}}$ is irreducible for any $1\leq i\leq r$ and the representation $\mathcal{T} \otimes \mathscr{O}_f$ is isomorphic to the direct sum of $\rho_1, \cdots, \rho_r$. To establish the claim by induction, it suffices to consider the case when $\rho\otimes \overline{\mathcal{K}}$ is reducible, which we assume. Since $\mathcal{T} \otimes \overline{\mathcal{K}}$ is semisimple, there exists a finite extension $\mathscr{K}$ of $\mathcal{K}$ contained in $\overline{\mathcal{K}}$ such that $\mathcal{T} \otimes \mathscr{K}$ decomposes into the direct sum of nonzero $\mathscr{K}$-linear subspaces $\mathscr{V}_1, \mathscr{V}_2$ stable under the action of $G_F$. Denote the $G_F$-representation $\mathcal{T} \otimes_\mathcal{O} \mathscr{O}$ by $\mathscr{T}$. Then the $\mathscr{O}$-module $\mathscr{T}_i:= \mathscr{T}\cap \mathscr{V}_i$ is stable under the action of $G_F$. Let $\mathscr{B}_i$ be a subset of $\mathscr{T}_i$ such that $\mathscr{B}_i$ is a basis of $\mathscr{V}_i$ over $\mathscr{K}$. Let $\mathscr{T}'_i$ denote the $\mathscr{O}$-submodule of $\mathscr{T}$ generated by $\mathscr{B}_i$. Note that $\mathscr{T}'_i$ is contained in $\mathscr{T}_i$ and is free over $\mathscr{O}$. For any element $t_i$ of $\mathscr{T}_i$, there exists a nonzero element $d_i \in \mathscr{O}$ such that $d_it_i$ belongs to $\mathscr{T}'_i$. Since $\mathscr{O}$ is Noetherian, $\mathscr{T}_i$ is finitely generated over $\mathscr{O}$ and hence there exists a nonzero element $f_i \in \mathscr{O}$ such that $f_i\mathscr{T}_i$ is contained in $\mathscr{T}'_i$. This proves that $\mathscr{T}'_i\otimes \mathscr{O}_{f_i}$ is equal to $\mathscr{T}_i\otimes \mathscr{O}_{f_i}$. Similarly, the direct sum of $\mathscr{T}'_1, \mathscr{T}'_2$ is equal to $\mathscr{T}$ after some localization. Hence there exists a nonzero element $f\in \mathscr{O}$ and representations of $G_F$ on two free $\mathscr{O}_f$-modules whose direct sum is isomorphic to $\mathscr{T}_f$ and whose tensor products with $\mathscr{K}$ are isomorphic to $\mathscr{V}_1, \mathscr{V}_2$. Then the claim follows from induction. 

By \cite[Lemma 7.2.2]{ChenevierGLn}, there exists a nonzero element $g\in \mathscr{O}_f$ such that for any ring homomorphism $\mu: \mathscr{O}_f \to \overline{\mathbb{Q}}_p$ with $\mu (g) \neq 0$, the representations $\mu \circ \rho_1, \cdots, \mu \circ \rho_r$ are absolutely irreducible. Replacing $g$ by its product with a power of $f$ (if necessary), we assume that $g$ is a nonzero element of $\mathscr{O}$. So for any map $\mu: \mathscr{O} \to \overline{\mathbb{Q}}_p$ with $\mu(fg)\neq 0$, the representations $\mu \circ \rho_1, \cdots, \mu \circ \rho_r$ are absolutely irreducible. For each ring homomorphism $\lambda: \mathcal{O} \to \overline{\mathbb{Q}}_p$, fix an extension $\tilde \lambda: \mathscr{O} \to \overline{\mathbb{Q}}_p$ of $\lambda$. Let $X^s + a_{s-1}X^{s-1} + \cdots + a_0$ be a monic polynomial over $\mathcal{O}$ of least degree which is satisfied by $fg$. Note that $a_0$ is nonzero. Since the representation $\mathscr{T}_f$ is isomorphic to the direct sum of $\rho_1, \cdots, \rho_r$, for any map $\lambda:\mathcal{O} \to \mathbb{Z}_p$ satisfying $\lambda(h)\neq 0$ where $h:= a_0$, the representation $\lambda \circ \rho$ is isomorphic to the direct sum of the absolutely irreducible representations $\tilde \lambda \circ \rho_1, \cdots, \tilde \lambda \circ \rho_r$, and hence $\lambda \circ \rho$ is semisimple. 
\end{proof}

\begin{proof}[Proof of Theorem \ref{Thm}]
Since $\rho$ is continuous with coefficients in $\mathcal{O}$ and $\mathcal{O}$ is a power series ring with coefficients in a $p$-adic integer ring or an affinoid algebra over a $p$-adic number field, by Grothendieck's monodromy theorem (\cite[pp. 515--516]{SerreTate}, \cite[Lemma 7.8.14]{BellaicheChenevierAsterisQUE}), the inertia group at any finite place of $F$ not dividing $p$ acts potentially unipotently on $\rho$. For each place $v\nmid p$ of $F$ such that $\rho$ has nonzero monodromy at $v$, choose a nonzero entry $\eta_v$ of the monodromy. By Proposition \ref{Prop}, there exists a continuous $\mathbb{Z}_p$-algebra homomorphism $\lambda: \mathcal{O} \to \overline{\mathbb{Q}}_p$ such that $\lambda(h)$ is nonzero and $\lambda (\eta_v)$ is nonzero for any place $v\nmid p$ of $F$ where $\rho$ has nonzero monodromy. Suppose $v\nmid p$ is a place of ramification of $\rho$. If $\rho$ has nontrivial monodromy at $v$, then $\lambda \circ \rho$ also has nontrivial monodromy at $v$ and hence ramified at $v$. If $\rho$ has trivial monodromy at $v$, then some element of the inertia group at $v$ acts on $\rho$ by a nontrivial finite order automorphism and hence has a nontrivial root of unity as an eigenvalue. So its induced action on $\lambda \circ \rho$ also has a nontrivial root of unity as an eigenvalue, which implies that $\lambda \circ \rho$ is ramified at $v$. Consequently, the set of places of ramification of $\rho$ and $\lambda \circ \rho$ are equal. Since $\lambda (h)$ is nonzero, the representation $\lambda \circ \rho$ is semisimple. By \cite[Theorem 1]{KhareRajanRamification}, the set of places of ramification of $\lambda \circ \rho$ is of density zero. So the set of places of ramification of $\rho$ is of density zero. This proves Theorem \ref{Thm}. 
\end{proof}

\begin{proof}[Proof of Theorem \ref{Thm Pure}]
Since the kernels of pure specializations of $\rho$ form a Zariski-dense subset of $\mathrm{Spec} \mathcal{O}$, by Proposition \ref{Prop semisimple}, there exists a continuous $\mathbb{Z}_p$-algebra homomorphism $\lambda: \mathcal{O} \to \mathbb{Z}_p$ such that $\lambda \circ \rho$ is semisimple and pure. Since $\rho$ is continuous with coefficients in a complete local Noetherian ring over $\mathbb{Z}_p$ with finite residue field of characteristic zero, by Grothendieck's monodromy theorem \cite[pp. 515--516]{SerreTate}, the inertia group at any finite place of $F$ not dividing $p$ acts potentially unipotently on $\rho$. Then the purity of $\lambda \circ \rho$ implies that the set of places of ramification of $\rho$ and $\lambda \circ \rho$ are equal by \cite[Theorem 3.1]{BigPuritySubAIF}. Since $\lambda \circ \rho$ is semisimple, by \cite[Theorem 1]{KhareRajanRamification}, the set of places of ramification of $\lambda\circ \rho$ is of density zero. So the set of places of ramification of $\rho$ is of density zero. This proves Theorem \ref{Thm Pure}. 
\end{proof}

\section*{Acknowledgement}
The author would like to acknowledge the support provided by a postdoctoral fellowship at the Ben-Gurion University of the Negev, offered by the Israel Science Foundation Grant number 87590011 of Prof.\,Ishai Dan-Cohen. He would like to thank Somnath Jha for suggesting improvements on an earlier draft of this work. 

\def\cprime{$'$} \def\Dbar{\leavevmode\lower.6ex\hbox to 0pt{\hskip-.23ex
  \accent"16\hss}D} \def\cfac#1{\ifmmode\setbox7\hbox{$\accent"5E#1$}\else
  \setbox7\hbox{\accent"5E#1}\penalty 10000\relax\fi\raise 1\ht7
  \hbox{\lower1.15ex\hbox to 1\wd7{\hss\accent"13\hss}}\penalty 10000
  \hskip-1\wd7\penalty 10000\box7}
  \def\cftil#1{\ifmmode\setbox7\hbox{$\accent"5E#1$}\else
  \setbox7\hbox{\accent"5E#1}\penalty 10000\relax\fi\raise 1\ht7
  \hbox{\lower1.15ex\hbox to 1\wd7{\hss\accent"7E\hss}}\penalty 10000
  \hskip-1\wd7\penalty 10000\box7}
  \def\polhk#1{\setbox0=\hbox{#1}{\ooalign{\hidewidth
  \lower1.5ex\hbox{`}\hidewidth\crcr\unhbox0}}}

\end{document}